\documentclass[10.5pt]{amsart}
\usepackage{a4,amssymb,amsfonts,amsmath,enumerate,color,url,hyperref,mathbbol}
\def\:{\thinspace:\thinspace}
\numberwithin{equation}{section}
\newtheorem{theo}{Theorem}

\newtheorem{lemma}[theo]{Lemma}
\newtheorem{prop}[theo]{Proposition}
\newtheorem{cor}[theo]{Corollary}
\newtheorem{defi}[theo]{Definition}

\theoremstyle{definition}

\newtheorem{rem}[theo]{Remark}

\DeclareMathOperator{\Id}{Id}

\numberwithin{theo}{section}

 \author{Delio Mugnolo}
 
 \address{Delio Mugnolo, Institut f\"ur Analysis, Universit\"at Ulm, 89069 Ulm, Germany}
\email{delio.mugnolo@uni-ulm.de}

\author{Serge Nicaise}

\address{Serge Nicaise, Universit\'e de Valenciennes et du Hainaut Cambr\'esis, LAMAV, FR CNRS 2956, ISTV, F-59313 - Valenciennes Cedex 9, France}
\email{Serge.Nicaise@univ-valenciennes.fr}

\thanks{This article has been initiated during a visit of the first author in Valenciennes and continued during a visit of the second author in Ulm. The authors thank  the Universit\'e de Valenciennes et du Hainaut Cambr\'esis and the Universit\"at Ulm for their hospitality. Both visits were  financially supported by the Land Baden--W\"urttemberg in the framework of the \emph{Juniorprofessorenprogramm} -- research project on ``Symmetry methods in quantum graphs''. We thank Fritz Gesztesy and Konstantin Pankrashkin for interesting discussions on the topic of Section~\ref{sec:krein}.}

\title[Heat equation under moment conditions]{The heat equation under conditions\\ on the moments in higher dimensions}


\begin{document}
\maketitle

\begin{abstract}
We consider the heat equation on the $N$-dimensional cube $(0,1)^N$ and impose different classes of integral conditions, instead of usual boundary ones. Well-posedness results for the heat equation under the condition that the moments of order 0 and 1 are conserved had been known so far only in the case of $N=1$ -- for which such conditions can be easily interpreted as conservation of mass and barycenter. In this paper we show that in the case of general $N$ the heat equation with such integral conditions is still well-posed, upon suitably relax the notion of solution. 
Existence and uniqueness of solutions with general initial data in a suitable space of distibutions over $(0,1)^N$ are proved by introducing two appropriate realizations of the Laplacian and checking by form methods that they generate analytic semigroups. The solution thus obtained does however solve the heat equation only in a certain distributional sense.
However, it turns out that one of these realizations is tightly related to a well-known object of operator theory, the Krein--von Neumann extension of the Laplacian. This connection also establishes well-posedness in a classical sense, as long as the initial data are $L^2$-functions.
\end{abstract}

\section{Introduction}

Fifty years ago, J.R.\ Cannon has suggested in~\cite{Can63} that realistic mechanical considerations suggest to study diffusion equations imposing a condition on the moment of order 0 of the unknown (i.e., on the total mass of the system), thus dropping one of the boundary conditions. Cannon's analysis was limited to the one-dimensional case, both in~\cite{Can63} and his later research surveyed in~\cite{Can84}. His original setting has been significantly generalized over the years (cf.~\cite[\S~1]{MugNic13} for a historical overview), in particular replacing also  the remaining boundary condition by a condition on the moment of order 1 (i.e., on the barycenter). 

Still, to the best of our knowledge only a few tentative extensions of the above conditions for heat or wave equations on higher dimensional domains have been proposed in the literature: We mention~\cite{MesBou01,MerBou03,Pul07}, where earlier related references are also collected. A first difficulty is that in case of a bounded domain in dimension $N\ge 2$ (unlike in the 1-dimensional case) infinitely many (boundary) conditions are necessary to determine a solution, while prescribing the moments of order 0 and 1 yields only finitely many conditions.

In this note, we introduce a general setting which in our opinion yields the proper $N$-dimensional extensions of the 1-dimensional moment conditions studied in~\cite{BouBen96,MugNic13}. Because the moments of order 0 and 1 of a given function are its integrals against the two vectors that span the space of one-dimensional harmonic functions, it is natural to conjecture that for $N\ge 2$ Cannon's setting should be extended by imposing orthogonality to the harmonic functions. This allows to define two diffusion-type operators $A$ and $\tilde{A}$ and to study the associated parabolic problem.

We stress that neither of these operators is a classical Laplacian: Rather, each of them is an isomorphic image of (a realization of) the Laplacian in a suitable space of distributions of $H^{-1}$-type.
An analogous problem has been already observed in~\cite{MugNic13}. Roughly speaking, $A$ turns out to be an operator {with conditions on the moment of order 0 along with uncountably many further conditions (that however impose in particular that the $N$ linear moments of order 1 of the unknown have to vanish); whereas} $\tilde{A}$ will be an operator with periodic boundary conditions and a condition on the moment of order 0 only. This is discussed in detail in Section~\ref{sec:homog}.
Nevertheless, $A$ and $\tilde{A}$ agree with the Laplacian in the sense of distributions.

In the special case of $N=1$ it has been shown in~\cite[\S\S~3--4]{MugNic13} that the solution of the abstract Cauchy problems associated with either of them also solves the classical heat equation. Indeed, the orbits of the semigroups generated by $A$ and $\tilde{A}$ contain only elements whose smoothness suffices to
ensure that the generator acts on them as the usual second derivative: This can be proved carefully describing the smoothing enhancement yielded by analyticity of the semigroups.

It seems that a similar strategy is not successful in the present  higher dimensional context. In fact, a direct computation suggests that in general the elements of the orbits are not much better than $L^2$. Nevertheless, we will eventually show in Section~\ref{sec:parab} that the semigroups actually solve the usual heat equation (with moment conditions) -- at least in a suitably weak, distributional sense.

Finally, in Section~\ref{sec:krein} we comment on the well-posedness of the relevant diffusion problem in a more usual $L^2$-setting. We also emphasize the connections between our investigations and Krein's theory of self-adjoint extensions: Also this observation seems to be new.

\section{The functional  setting}\label{sec:gelfand}

Let $\Omega\subset \mathbb R^N$ be an open domain with Lipschitz boundary. It is well-known (see e.g.~\cite[Thm.~2.5]{AreEls12}) that 
\begin{equation}\label{eq:decomposition}
H^1(\Omega)=H^1_0(\Omega)\oplus {\rm Har}(\Omega),
\end{equation}
where ${\rm Har}(\Omega)$ denotes the space of {(weakly) harmonic} functions, i.e.,
\[
{\rm Har}(\Omega):=\left\{u\in H^1(\Omega):\int_\Omega \nabla u\overline{\nabla v}dx=0 \ \forall v\in H^1_0(\Omega) \right\}.
\]
Furthermore, the trace operator $\gamma_0$ is an isomorphism between ${\rm Har}(\Omega)$ (the orthogonal of its null space) and $H^\frac{1}{2}(\partial \Omega)$ (its range).

If $T^N$ denotes the $N$-dimensional torus, we can consider the (Hilbert) space $H^1(T^N)$, defined as usual as the space of periodic functions defined on $T^N$. Throughout this article we will denote
\[
\mu_0(h):=\langle h,1\rangle,\qquad h\in H^{-1}(T^N),
\]
the mean -- i.e., the moment of order 0 -- of a distribution  $h$.

Now, $H^1_0((0,1)^N)$ is a closed subspace of $H^1(T^N)$, hence we can consider its orthogonal complement which, as above, turns out to be
\begin{equation}\label{eq:decomposition-2}
{\rm Har}(T^N):=\left\{u\in H^1(T^N):\int_{T^N} \nabla u\overline{\nabla v}dx=0 \ \forall v\in H^1_0(\Omega) \right\}=:H^1(T^N)\ominus H^1_0\left( (0,1)^N \right).
\end{equation}
Hence, using $L^2\left( (0,1)^N \right)$ as pivot space we can find the decomposition
\begin{equation}\label{eq:decomposition-3}
H^{-1}(T^N)=H^{-1}\left( (0,1)^N \right)\oplus\left(	{\rm Har}(T^N)\right)',
\end{equation}
where $H^{-1}\left( (0,1)^N \right):=(H^1_0(0,1)^N)'$.

\begin{rem}\label{rem:1-dim-comp-2}
In other words, ${\rm Har}(T^N)$ is the space of periodic $H^1$-functions that are {(weakly) harmonic}, whereas ${\rm Har}\left( (0,1)^N \right)$ is the space of all {(weakly) harmonic} $H^1$-function on the unit cube. It is by definition clear that ${\rm Har}(T^N)$ is a subspace of ${\rm Har}\left( (0,1)^N \right)$. In fact, it will in general be a proper subspace: For $N=1$ we e.g.\ find that ${\rm Har}(T^N)$ and ${\rm Har}\left( (0,1)^N \right)$ are isomorphic to the spaces ${\mathbb P}_0$ and  ${\mathbb P}_1$ of polynomials of degree $0$ and of degree less than or equal to $1$, respectively. 
\end{rem}

Let us denote by $\Id$ the orthogonal projector of $H^{-1}(T^N)$ onto its closed subspace $H^{-1}\left( (0,1)^N \right)$ and by $\Id_m$ the restriction of $\Id$ to the annihilator 
$$H:=\{w\in H^{-1}(T^N)\, :\, \langle w, v\rangle =0\ \forall v\in {\rm Har}(T^N)\}$$
of ${\rm Har}(T^N)$, which is an isomorphism from $H$ to $H^{-1}\left( (0,1)^N \right)$.

Denote by
$$
H^{\frac12}(\partial T^N):=\{\gamma_0 v: v\in H^1\left( T^N\right)\},
$$
where $\gamma_0$ is here the trace operator from $H^1\left((0,1)^N\right)$ into $H^{\frac12}(\partial (0,1)^N)$.
The space $H^{\frac12}(\partial T^N)$ is smaller than $H^{\frac12}(\partial (0,1)^N)$ due to the periodicity assumption on elements from $H^1\left( T^N\right)$. Nevertheless $
H^{\frac12}(\partial T^N) 
$ is a Hilbert space with the induced norm
$$
\|\varphi\|_{H^{\frac12}(\partial T^N)}:=\inf_{\substack{v\in H^1\left( T^N\right)\\ \varphi=\gamma_0 v}} \|v\|_{H^1\left( T^N\right)}.
$$
\begin{lemma}\label{lemmafctharmonique}
For all $\varphi\in H^{\frac12}(\partial T^N)$, there exists a unique $R\varphi\in {\rm Har}(T^N)$ such that
\[
\gamma_0 R\varphi =\varphi  \hbox{ on }\partial (0,1)^N.
\]
In other words, ${\rm Har}(T^N)$ is isomorphic to $H^{\frac12}(\partial T^N)$.
\end{lemma}
The existence of the  right inverse of the trace operator, seen as an operator from the space of weakly harmonic functions to the space of $H^\frac12$-functions over the boundary of a domain, is classical - see e.g.\ the much more general discussion in~\cite[\S~2.7]{LioMag72} or the alternative approach in~\cite[Lemma~1.2]{Gre87}.

Define 
\begin{eqnarray*}
\tilde V:=\{f\in L^2\left((0,1)^N\right)\, : \, (f|g)=0\ \forall g\in {\rm Har}(T^N)\},
\\
V:=\{f\in L^2\left((0,1)^N\right)\, : \, (f|g)=0\ \forall g\in {\rm Har}\left( (0,1)^N \right)\},
\end{eqnarray*}
that are clearly two closed subspaces of $L^2\left((0,1)^N\right)$. Indeed by introducing
$V_0$ and $V_1$ as the closure of ${\rm Har}(T^N)$  and ${\rm Har}\left( (0,1)^N \right)$  in $L^2\left((0,1)^N\right)$, respectively,
\begin{eqnarray*}
V_0&:=&\overline{{\rm Har}(T^N)}^{_{L^2}}\\
V_1&:=&\overline{{\rm Har}\left((0,1)^N\right)}^{_{L^2}}
\end{eqnarray*}
one sees that $\tilde{V}$ and $V$ are the $L^2$-orthogonal complements of $V_0$ and $V_1$, respectively -- shortly:
\[
\tilde V=V_0^\perp\quad \hbox{and}\quad V=V_1^\perp.
\]
As ${\rm Har}(T^N)\subset {\rm Har}\left( (0,1)^N \right)$ and hence $V_0\subset V_1$,  we conclude that 
\[
V\subset \tilde V.
\] 

We are going to need below a certain characterization of $V_0$.
For that purpose, we can notice that
$V_0$ is trivially included into the domain  of $\Delta_{L^2}$ --  the part of $\Delta$ in $L^2\left((0,1)^N\right)$ --, i.e., into
\[
D(\Delta_{L^2}):=\left\{v\in L^2\left((0,1)^N\right): \Delta v\in L^2\left((0,1)^N\right)\right\},
\]
which is a Hilbert space in its own right whenever equipped with the natural graph norm
\[
v\mapsto \left(\|v\|_{L^2\left((0,1)^N\right)}^2+\|\Delta v\|_{L^2\left((0,1)^N\right)}^2\right)^{\frac12}.
\]
It follows from the general theory of interior elliptic regularity that if $u\in D(\Delta_{L^2})$, then in general one only has $u\in 
H^2_{\rm loc}$, hence a trace of $u$ need not exist as an element of $L^2\left(\partial (0,1)^N\right)$. However, we can give a meaning of its trace on each face as a vector in a suitable space of distributions. Indeed following \cite[Thm 1.5.3.4]{Gri85},
 the space ${\mathcal D}([0,1]^N)$ -- the set of the restriction of elements of ${\mathcal D}({\mathbb R}^N)$ to $(0,1)^N$ -- is dense in $D(\Delta_{L^2})$
and for all $i=1,\cdots, n$, the trace operator
\[
\gamma_{i\pm}:   v\mapsto v_{\Gamma_{i\pm}},
\]
which is certainly defined for $v\in {\mathcal D}([0,1]^N)$, has a unique continuous extension from $D(\Delta_{L^2})$ into $(\tilde H^{\frac12}(\Gamma_{i\pm}))'$. 
We denote also this extension by $\gamma_{i\pm}$.
Here and below  we are denoting by $\Gamma_{i\pm}$ the faces of the hypercube, which are defined by
\begin{eqnarray*}
&&\Gamma_{i-}:=\{x\in [0,1]^N: x_i=0 \hbox{ and } x_j\in (0,1), \forall j\ne i\},\\
&&\Gamma_{i+}:=\{x\in [0,1]^N: x_i=1 \hbox{ and } x_j\in (0,1), \forall j\ne i\}.
 \end{eqnarray*}
Furthermore $\tilde H^{\frac12}(\Gamma_{i\pm})$ denotes the subspace of elements
$w\in H^{\frac12}(\Gamma_{i\pm})$ such that $\tilde w$, its extension  by zero outside $\Gamma_{i\pm}$, belongs to $H^{\frac12}\left(\partial (0,1)^N\right).$

\begin{defi}\label{def:weaklyper}
Let $s>\frac12$. We call  a  function $v\in D(\Delta_{L^2})\cup  H^s\left((0,1)^N\right)$ \emph{periodic} if it satisfies
\begin{equation}
\label{weakperiodic}
 \gamma_{i-}   v= T_i\gamma_{i+}   v \hbox{ in }\left(\tilde H^{\frac12}(\Gamma_{i+})\right)',\quad \forall i=1,\cdots, n .
\end{equation}
\end{defi}
Here $T_i:\left(\tilde H^{\frac12}(\Gamma_{i+})\right)'\to \left(\tilde H^{\frac12}(\Gamma_{i-})\right)'$  is the operator defined by means of
\[
\langle T_i \psi, \varphi\rangle:=\langle   \psi, T_i^* \varphi\rangle\ \forall \psi\in \left(\tilde H^{\frac12}(\Gamma_{i+})\right)',
\varphi\in \tilde H^{\frac12}(\Gamma_{i-}),
\]
with
\[
(T_i^* \varphi)(x_1, \cdots, 1,\cdots, x_n):=\varphi(x_1, \cdots, 0,\cdots, x_n)
\quad \hbox{ for a.e.\ } (x_1, \cdots, 1,\cdots, x_n)\in \Gamma_{i+}.
\]

\begin{lemma}\label{lemmaV0}
One has
\[
V_0\subset \left\{v\in L^2\left((0,1)^N\right): \Delta v= 0 \hbox{ in }
{\mathcal D}'\left((0,1)^N\right) \hbox{ satisfying }  \eqref{weakperiodic}\right\}.
\]
\end{lemma}
\begin{proof}
As ${\rm Har}(T^N)\subset H^1\left( T^N\right)$, any $u\in {\rm Har}(T^N)$ clearly satisfies
\eqref{weakperiodic}. Since there exists $C>0$ such that  
\[
\|\gamma_{i-}   v-T_i\gamma_{i+}v\|\leq C \left(\|v\|_{L^2\left((0,1)^N\right)}^2+\|\Delta v\|_{L^2\left((0,1)^N\right)}^2\right)^{\frac12}, \quad \forall v\in D\left(\Delta_{L^2}\right),
\]
we directly conclude that any $v\in V_0$ still satisfies \eqref{weakperiodic}.
\end{proof}

\begin{rem}\label{H1TNper}
 Note that in particular  any  function in $H^1(T^N)$ is  periodic in the sense of the above definition, and by Lemma~\ref{lemmaV0} so is any element of $H^1\left( T^N\right)+V_0$, too.
 \end{rem}

\section{An equivalent inner product in $H^{-1}(T^N)$}

\begin{lemma}\label{lemma1}
For all $f\in H^{-1}(T^N)$ there exists a unique $u_f\in H^{1}_m(T^N)$ such that 
\[
{\rm div } \nabla u_f=f-\mu_0(f).
\]
An equivalent inner product in $H^{-1}(T^N)$ is therefore given by
\begin{equation*}
(\nabla u_f|\nabla  u_g)_{L^2(T^N)^N}+\mu_0(f)\mu_0(\bar g),\qquad f,g\in H^{-1}(T^N).
\end{equation*}
\end{lemma}

In the proof we will need a closed subspace $
H^{1}_m(T^N)$ of $H^1\left( T^N\right)$ defined by
$$
H^{1}_m(T^N):=\left\{u\in H^{1}(T^N): \int_{(0,1)^N} u(x)dx=0\right\}.
$$
By the Poincar\'e-type inequality 
$$
\exists C>0: \|u\|_{L^2(T^N)}\leq C  \|\nabla u\|_{L^2(T^N)^N}, \forall u\in H^{1}_m(T^N),
$$
{
see e.g.~\cite[Prop.~5.44]{DemDem07}, which is based on the compactness of the embedding of  $H^{1}((0,1)^N)$ into $L^2((0,1)^N)$, cf.~\cite[Thm.~1.4.3.2]{Gri85}),
} we know that
\[
u\mapsto \|\nabla u\|_{L^2(T^N)^N}
\]
defines a norm on $H^{1}_m(T^N)$ that is equivalent to the standard $H^1$-norm.

\begin{proof}
Given $f\in H^{-1}(T^N)$, we set
$$
f_0=f-\langle f, 1\rangle,
$$
where here and below the duality bracket is between $H^{-1}(T^N)$ and $H^{1}(T^N)$. Hence 
$f_0\in H^{-1}(T^N)$ and satisfies
$$
\langle f_0, 1\rangle=0.
$$ 
Now by the Theorem of Riesz--Fr\'{e}chet there exists a unique solution $u_f\in H^{1}_m(T^N)$ of
\begin{equation}\label{serge1}
\int_{(0,1)^N} \nabla u_f\cdot \nabla \bar v dx=-\langle f_0, v\rangle\ \forall v\in H^{1}_m(T^N).
\end{equation}
Since $
\langle f_0, 1\rangle=0,
$
this identity remains valid on the whole $H^{1}(T^N)$, namely
\begin{equation}\label{serge2}
\int_{(0,1)^N} \nabla u_f\cdot \nabla \bar v dx=-\langle f_0, v\rangle\ \forall v\in H^{1}(T^N).
\end{equation}
By choosing smooth enough test functions $v$, we see that
$$
{\rm div } \nabla u_f=f_0 \hbox{ in } {\mathcal D}'\left((0,1)^N\right),$$
or equivalently
\begin{equation}
\label{lapldistr}
f={\rm div } \nabla u_f+\langle f, 1\rangle.
\end{equation}
According to (\ref{serge2}), we have
$$
\|\nabla u_f\|_{L^2(T^N)^N}^2=-\langle f_0, u_f\rangle\leq \|f_0\|_{H^{-1}(T^N)} \|u_f\|_{H^{1}(T^N)},
$$
and by the equivalence of norm mentioned before, we get
$$
\|\nabla u_f\|_{L^2(T^N)^N} \lesssim \|f_0\|_{H^{-1}(T^N)} \lesssim \|f\|_{H^{-1}(T^N)}.
$$
Conversely (\ref{serge2}) is equivalent to
$$
\langle f, v\rangle =-\int_{(0,1)^N} \nabla u_f\cdot \nabla \bar v dx+\langle f, 1\rangle \overline{\langle  v,1\rangle}\ \forall v\in H^{1}(T^N).
$$
Consequently  
\begin{eqnarray*}
\|f\|_{H^{-1}(T^N)}&=&\sup_{ \|v\|_{ H^{1}(T^N)}=1} |\langle f, v\rangle|
\\
&\leq& \sup_{ \|v\|_{ H^{1}(T^N)}=1} \left(\left|\int_{(0,1)^N} \nabla u_f\cdot \nabla \bar v dx\right|+|\langle f, 1\rangle| |\langle  v,1\rangle|\right)
\\
&\lesssim& \|\nabla u_f\|_{L^2(T^N)^N}+|\langle f, 1\rangle|.
\end{eqnarray*}
This completely proves the assertion.
\end{proof}

\begin{rem}
Let us briefly compare the theory we have just developed with that introduced in~\cite{MugNic13} in   the one-dimensional setting. We define the ``primitive'' $P_N f\in L^2(T^N)^N$ of any $f\in H^{-1}(T^N)$ by
$$
P_N f:=\nabla u_f +\frac{\mu_0(f)}{N} \left(\vec x-\vec{\frac12}\right),\quad \forall f\in H^{-1}(T^N)
$$
where $\vec x$ and $\vec{\frac12}$ denote the vector-valued  functions
$$
\vec x:(x_1,\cdots, x_N) \mapsto\begin{pmatrix}
x_1\\ \vdots\\x_N
\end{pmatrix}\qquad \hbox{and}\qquad 
\vec{\frac12}:(x_1,\cdots, x_N) \mapsto \frac12\begin{pmatrix}
1\\ \vdots\\1
\end{pmatrix},
$$
respectively. From this expression, we see that
$$
{\rm div } P_N f=f \hbox{ in } {\mathcal D}'\left((0,1)^N\right),$$
and
$$
\langle P_N f, \alpha\rangle=0\ \forall \alpha \in {\mathbb C}^N.
$$
It finally follows from Lemma~\ref{lemma1} that the inner product
\begin{equation*}
\label{newnorm}
(f|g)_{H^{-1}(T^N)}=(P_N f|P_Ng)_{L^2(T^N)}+\mu_0(f)\mu_0(\bar g),
\end{equation*}
induces a norm equivalent to the standard norm of $H^{-1}(T^N)$.
\end{rem}

Note further that if $f\in L^2(T^N)$, then the solution  $u_f\in H^{1}_m(T^N)$ of
(\ref{serge1}) belongs to $H^2(T^N)$ and therefore $P_N f$ belongs to $H^{1}(T^N)^N$.
\begin{defi}
We define an operator $L$ from $H^1\left((0,1)^N\right)$ to $H$ by 
\[
L f:={\rm Id}_m^{-1} (\Delta f).
\]
\end{defi}
Observe that $L$ is well-defined because $\Delta f\in H^{-1}\left((0,1)^N\right)$ for all $f\in H^1\left((0,1)^N\right)$.

\begin{rem}\label{rem_proof_Vdense}
If in particular $f\in H^1\left( T^N\right)\cap H^2\left((0,1)^N\right)$ satisfying  Neumann boundary conditions, then $Lf=\Delta f\in L^2\left((0,1)\right)^N \subset H^{-1}\left((0,1)^N\right)$ and  hence in particular ${\rm Id}_m^{-1}(\Delta f)\in H^{-1}_m\left( T^N\right)$. 
Hence there exists by Lemma~\ref{lemma1} a unique $u_{L f}\in H^1_m(T^N)$ such that
\[
\Delta u_{L f}=L f-\mu_0(L f),
\]
that is,
\[
\Delta u_{L f}=L f.
\]
Because $f-\mu_0(f)\in H^1_m(T^N)$, one can easily conjecture that
\begin{equation}
\label{udeltaf=fmuf}
u_{L f}=f-\mu_0(f).
\end{equation}
This is indeed the case, as
\begin{eqnarray*}
\int_{(0,1)^N}\nabla (f-\mu_0(f))\cdot \nabla \bar v\,dx&=&\int_{(0,1)^N}\nabla f \cdot \nabla \bar v\,dx\\
&=&\int_{\partial (0,1)^N} \frac{\partial f}{\partial n}\bar v dx-\int_{(0,1)^N}\Delta f   \bar v\,dx\\
&=& -\int_{(0,1)^N}\Delta f   \bar v\,dx,\quad \forall v\in H^1_m(T^N),
\end{eqnarray*}
by the Gau{\ss}--Green formula.
\end{rem}


\begin{rem}\label{idmdl}
Note that our definition  implies directly that 
\[
\mu_0(L f)=\langle {\rm Id}_m^{-1}\Delta f, 1\rangle=0,\quad \forall f\in H^1\left((0,1)^N\right),
\] 
and that ${\rm Id}_m^{-1}\Delta f= \Delta f$ in $ {\mathcal D}'\left((0,1)^N\right)$.
\end{rem}

\begin{theo}\label{vh}
The space $V$ and hence $\tilde{V}$ are densely and compactly embedded in $H$.
\end{theo}
\begin{proof}
Denote by 
$
\tilde H$ the closure of $V$ in $H$.
To prove that
$$
H=\tilde H,
$$
it then suffices to show that any $f\in H$ orthogonal to $V$ is zero. Let $f\in H$ be orthogonal to $V$, i.e.,
\[
(f|g)_{H^{-1}(T^N)}=0\ \forall g\in V.
\]
As $\mu_0(f)=\mu_0(g)=0$, we deduce that
$$
\int_{(0,1)^N}\nabla u_f\cdot \nabla \bar u_g\,dx= 0\ \forall g\in V.
$$
According to \eqref{serge1}, we get equivalently
$$
\langle f, u_g\rangle= 0\ \forall g\in  V.
$$
But we will show below that
\begin{equation}\label{density21/12}
\{u_g: g\in V\}+ {\rm Har}(T^N)\quad
\hbox{ 
is dense in } H^1\left( T^N\right),
\end{equation}
and therefore as $f\in H$, we deduce that $f=0$ because
$$
\langle f, v\rangle= 0\ \forall v\in H^1\left( T^N\right).
$$
It remains to prove \eqref{density21/12}. For $u\in H^1\left( T^N\right)$, we have already noticed in Lemma \ref{lemmafctharmonique} that
$
u-R(\gamma_0 u)\in H^1_0\left((0,1)^N\right)$, therefore there exists a sequence of $\psi_n\in {\mathcal D}\left((0,1)^N\right)$ such that
\[
\psi_n\to u-R(\gamma_0 u) \hbox{ in } H^1_0\left((0,1)^N\right) \hbox{ as } n\to \infty,
\]
or equivalently
\begin{equation}\label{density21/12b}
\psi_n+R(\gamma_0 u)\to u \hbox{ in } H^1\left( T^N\right) \hbox{ as } n\to \infty.
\end{equation}

Now for $\psi\in {\mathcal D}\left((0,1)^N\right)$, we notice that
$
L\psi =\Delta \psi\in V$, i.e., $\Delta \psi$ is orthogonal to the {(weakly) harmonic} functions: In fact if $v$ is a {(weakly) harmonic} function, then exploiting the fact that $\psi$ has compact support the Gau{\ss}--Green formula yields
\[
\int_{(0,1)^N} \Delta \psi \bar v dx=-\int_{(0,1)^N}\nabla \psi \nabla \bar v dx=0,
\]
where the last identity follows by definition of {(weakly) harmonic} function. Furthermore, $u_{\Delta \psi}=\psi-\mu_0(\psi)$ by Remark~\ref{rem_proof_Vdense}.

Applying these last remarks to~\eqref{density21/12b} means that
\[
u_{\Delta \psi_n}+\mu_0(\psi_n)+R(\gamma_0 u)\to u \hbox{ in } H^1\left( T^N\right) \hbox{ as } n\to \infty.
\]
As $\mu_0(\psi_n)+R(\gamma_0 u)\in {\rm Har}(T^N)$ and $\Delta \psi_n\in V$, the density assertion in
\eqref{density21/12} follows. Finally, compactness follow by the compactness of the embedding $H^1\left( T^N\right)\hookrightarrow L^2\left((0,1)^N\right)$.
\end{proof}

We are finally in the position to prove a formula that can be seen as an $H^{-1}$-analogue of the usual Gau{\ss}--Green-formula that hold with respect to the inner product of $L^2$.

\begin{lemma}\label{lemma0.4}
For all $f \in H^1\left( T^N\right)$ and all $h\in L^2\left((0,1)^N\right)$ one has
\begin{equation}\label{lemma0.4id}
({\rm Id}_m^{-1}(\Delta f)|h)_{H^{-1}(T^N)} =-(f|h)_{L^2(T^N)}+ (R\gamma_0 f|h)_{L^2(T^N)}.
\end{equation}
\end{lemma}
\begin{proof}
For all $f \in H^1\left( T^N\right)$, as $\Delta f$ belongs to $H^{-1}\left((0,1)^N\right)$, we can set
$g={\rm Id}_m^{-1}(\Delta f)$ that satisfies $\mu_0(g)=0$  and  consequently 
for all $h\in L^2\left((0,1)^N\right)$ one has
\begin{equation}\label{serge20/12}
(g|h)_{H^{-1}(T^N)}  =\int_{(0,1)^N}\nabla u_g\cdot  P_N h\, dx,
\end{equation}
with
$$
 P_N h=\nabla u_h+\frac{\mu_0(h)}{N} \left(\vec x-\vec{\frac12}\right).
$$
Owing to \eqref{serge1}, we get
$$
\int_{(0,1)^N}\nabla u_g\cdot  \nabla u_h\, dx=-\langle g, u_h\rangle,
$$
and by the definition of ${\rm Id}_m^{-1}$, we obtain
$$
\int_{(0,1)^N}\nabla u_g\cdot  \nabla u_h\, dx=-\langle \Delta f, u_h-R\gamma_0 u_h\rangle
=-\langle \Delta (f-R\gamma_0 f), u_h-R\gamma_0 u_h\rangle.
$$
Hence by the definition of $\Delta (f-R\gamma_0 f)$ as element of $H^{-1}\left((0,1)^N\right)$, we get
\begin{eqnarray*}
\int_{(0,1)^N}\nabla u_g\cdot  \nabla u_h\, dx
&=&\langle \nabla (f-R\gamma_0 f)\cdot \nabla (u_h-R\gamma_0 u_h)\, dx
\\
&=&-\int_{(0,1)^N}  (f-R\gamma_0 f)  \Delta (u_h-R\gamma_0 u_h)\,dx.
\end{eqnarray*}
As $R\gamma_0 u_h$ is harmonic and $\Delta u_h=h-\mu_0(h)$, we get
\begin{eqnarray}\label{serge20/12b}
&&\int_{(0,1)^N}\nabla u_g\cdot  \nabla u_h\, dx
= 
-\int_{(0,1)^N}  (f-R\gamma_0 f)  (h-\mu_0(h))\,dx
\\
&&\hspace{1cm}=
-\int_{(0,1)^N}  f  h\,dx+\int_{(0,1)^N}   R\gamma_0 f   h\,dx 
+\mu_0(h) \int_{(0,1)^N}  (f-R\gamma_0 f) \,dx.
\nonumber
\end{eqnarray}
On the other hand,
one sees that
$$
\int_{(0,1)^N}\nabla u_g\cdot  \left(\vec x-\vec{\frac12}\right) \, dx=
\int_{(0,1)^N}\nabla u_g\cdot  \nabla v\, dx,
$$
where
$v\in H^1\left( T^N\right)$ is defined by
$$
v(x):=\frac12 \left\|\vec x-\vec{\frac12}\right\|^2_{L^2}.
$$
Hence \eqref{serge2} yields
$$
\int_{(0,1)^N}\nabla u_g\cdot  \left(\vec x-\vec{\frac12}\right) \, dx=-\langle g, v\rangle,
$$
and again by definition of ${\rm Id}_m^{-1}$,
$$
\int_{(0,1)^N}\nabla u_g\cdot  \left(\vec x-\vec{\frac12}\right) \, dx=-\langle \Delta f, v-R\gamma_0 v\rangle
=-\langle \Delta (f-R\gamma_0 f), v-R\gamma_0 v\rangle.
$$
As before we then obtain
$$
\int_{(0,1)^N}\nabla u_g\cdot  \left(\vec x-\vec{\frac12}\right) \, dx=-\langle f-R\gamma_0 f, \Delta (v-R\gamma_0 v)\rangle=-N \langle f-R\gamma_0 f, 1\rangle.
$$
This identity and \eqref{serge20/12b} in \eqref{serge20/12} yields the conclusion.
\end{proof}


\begin{rem}
Let us comment on the special case $N=1$. Then for all $h\in L^2(0,1)$ $P_1h$ is equal to $Ph$ defined as in~\cite[(2.5)]{MugNic13}. 
Furthermore, also in view of Remark~\ref{rem:1-dim-comp-2} we find that $V$ and $\tilde{V}$ agree with the two spaces with same name introduced in~\cite[\S~2]{MugNic13}.
Therefore, the theory we develop in the present paper is a generalization of that introduced in~\cite{MugNic13}.
\end{rem}

\section{Operators in the space of zero mean functions}\label{sec:homog}

In this section we want to determine precisely  two relevant realizations of the Laplacian in $H$, the space of those functionals that annihilate periodic and {(weakly) harmonic} functions.

We are going to consider the sesquilinear form $a$ defined by
\begin{equation}
\label{definitionform}
a(f,g):=\int_{(0,1)^N} f(x) \overline{g}(x)\,dx,
\end{equation}
with form domain either  $\tilde{V}$ or $V$.
Since both $\tilde V$ and $V$ are dense in $H$ by Lemma~\ref{vh},  the form $a$ with domain  $\tilde{V}$ or $V$ is associated with a linear operator  $(\tilde{A},D(\tilde{A}))$ or $(A,D(A))$, respectively, defined by
\begin{eqnarray*}
D(\tilde{A})&:=&\left\{f\in \tilde{V}: \exists g\in H: a(f,h)=\int_{(0,1)^N} \nabla u_g(x)\cdot \nabla \bar u_h(x)\,dx\;\;\forall h\in \tilde{V}\right\},\\
\tilde{A} f&:=&g
\end{eqnarray*}
and
\begin{eqnarray*}
{D}(A)&:=&\left\{f\in V: \exists g\in H: a(f,h)=\int_{(0,1)^N} \nabla u_g(x)\cdot \nabla \bar u_h(x)\,dx\;\;\forall h\in V\right\},\\
A f&:=&g.
\end{eqnarray*}
Let us describe these two operators more precisely.

\begin{theo}
\label{identk=0}
One has
\begin{eqnarray*}
D(\tilde{A})&=& \tilde V\cap (H^1_m(T^N)+V_0),\\
 \tilde{A} f&=&-{\rm Id}_m^{-1} \Delta f,\quad \forall f\in D(\tilde{A}).
\end{eqnarray*}
\end{theo}
Observe that in view of Remark~\ref{H1TNper}, each function in $D(\tilde{A})$ is weakly periodic.

\begin{proof}
Denote
$${\mathcal K}:= \tilde V\cap (H^1_m(T^N)+V_0).$$
Let us first show the inclusion $D(\tilde{A})\subset {\mathcal K}$. Let $f\in D(\tilde{A})$. Then there exists $g=Af\in H^{-1}_m(T^N)$   (because  $g\in H$)   such that
$$
\int_{(0,1)^N} f(x) \overline{h}(x)\,dx =\int_{(0,1)^N} \nabla u_g(x)\cdot \nabla \bar u_h(x)\,dx,\qquad\forall h\in \tilde V.
$$
But according to (\ref{serge1}) we then have equivalently 
\begin{eqnarray*}
\int_{(0,1)^N} f(x) \overline{h}(x)\,dx=-\overline{\langle h, u_g\rangle}=
 -\int_{(0,1)^N} \bar h(x)   u_g(x)\,dx,\qquad\forall h\in \tilde V,
\end{eqnarray*}
because $h$ belongs to $L^2\left((0,1)^N\right)$. This means equivalently that
$
f+u_g$ is orthogonal (in the $L^2$ sense) to $\tilde V$, i.e.,
$
f+u_g$ belongs to $V_0$. As elements of $V_0$ are harmonic functions, we deduce that
$$
\Delta (f+u_g)=0 \hbox{ in } {\mathcal D}'\left((0,1)^N\right),
$$
and by \eqref{lapldistr}, we find
$$
\Delta f=-g  \hbox{ in } {\mathcal D}'\left((0,1)^N\right).
$$
This shows that $\Delta f$ belongs to $H^{-1}\left((0,1)^N\right)$ and 
reminding that $g$ belongs to $H$, we get further
$$
g=-{\rm Id}_m^{-1} (\Delta f).
$$
Similarly as element of $V_0$ are weakly periodic and $u_g\in H^1\left( T^N\right)$ (hence weakly periodic), we deduce that $f$ is weakly periodic.

Let us now prove the converse inclusion. Let $f\in {\mathcal K}$ then $g=-{\rm Id}_m^{-1} (\Delta f)$ belongs to $H$  and by Lemma \ref{lemma0.4}, we get
 for any $h\in \tilde V$
\begin{eqnarray*}
(g|h)_H= (f|h)_{L^2}, \end{eqnarray*}
as 
$(R\gamma_0 f|h)_{L^2(T^N)}=0.$
This shows that $f\in D(\tilde{A})$ and concludes the proof.
\end{proof}

\begin{theo}
\label{identVk=0}
One has
\begin{eqnarray*}
{D}(A)&=& V\cap (H^1_m(T^N)+V_1),\\
 A f &=&-{\rm Id}_m^{-1} \Delta f,\quad \forall f\in D(A).
\end{eqnarray*}
\end{theo}
\begin{proof}
Denote
$${\mathcal K}:=V\cap (H^1_m(T^N)+V_1).$$
Let us first show the inclusion ${D}(A_0)\subset {\mathcal K}$. Let $f\in {D}(A_0)$. Then there exists $g=:A_0f\in H\equiv H^{-1}_m(T^N)$    such that
$$
\int_{(0,1)^N} f(x) \overline{h}(x)\,dx =\int_{(0,1)^N} \nabla u_g(x)\cdot \nabla \bar u_h(x)\,dx,\qquad\forall h\in  V.
$$
But according to (\ref{serge1}) we then have equivalently 
\begin{eqnarray*}
\int_{(0,1)^N} f(x) \overline{h}(x)\,dx=-\overline{\langle h, u_g\rangle}=
 -\int_{(0,1)^N} \bar h(x)   u_g(x)\,dx,\qquad\forall h\in  V,
\end{eqnarray*}
because $h$ belongs to $L^2\left((0,1)^N\right)$. This means equivalently that
$
f+u_g$ is orthogonal (in the $L^2$ sense) to $V$, i.e.,
$
f+u_g$ belongs to $V_1$. As elements of $V_1$ are harmonic functions, we deduce that
$$
\Delta (f+u_g)=0 \hbox{ in } {\mathcal D}'\left((0,1)^N\right),
$$
and by \eqref{lapldistr}, we find
$$
\Delta f=-g  \hbox{ in } {\mathcal D}'\left((0,1)^N\right).
$$
This shows that $\Delta f$ belongs to $H^{-1}\left((0,1)^N\right)$ and 
reminding that $g$ belongs to $H$, we get further
$$
g=-{\rm Id}_m^{-1} (\Delta f).
$$
This shows the  desired inclusion because by definition $ {D}(A_0)\subset V$.

The converse inclusion is proved as in the previous Theorem by using   Lemma \ref{lemma0.4}.
\end{proof}

\begin{rem}\label{rem45}
If one wants to study the Laplacian acting on functions on a more general domain $\Omega\subset \mathbb R^N$, one may look for a Lipschitz-homeomorphism mapping $(0,1)^N$ into $\Omega$. In this way, however, also the Laplacian on $(0,1)^N$ is transformed into a different elliptic operator on $\Omega$. The following is a possibly more convenient approach: Take a ``partition'' $\Gamma=\{\Gamma_1,\ldots,\Gamma_{2N}\}$ of $\partial \Omega$ such that
\begin{itemize}
\item each set $\Gamma_i\subset \partial \Omega$ is open,
\item the sets $\Gamma_i$   are pairwise disjoint,
\item  the union of their closures covers $\partial \Omega$, and
\item such that each $\Gamma_i$ is bijective (say, via some $\theta_i$) to $\Gamma_{N+i}$ for all $i=1,\ldots,N$.
\end{itemize}
Define 
\[
H^1_\Gamma(\Omega):=\{ u\in H^1((\Omega)): \gamma_i u =\gamma_{N+i} u \circ \theta_i\hbox{ on } \Gamma_i, \forall i=1,\ldots,N\},
\]
where $\gamma_i$ is the trace operator on $\Gamma_i$.
In this way, we can regard in a natural way the space $H^1\left( T^N\right)$ as the space $H^1_\Gamma\left( (0,1)^N \right)$ with the partition $\Gamma$ made of faces of the hypercube, and deduce that all the constructions of this paper can be extended to consider realizations (with linear conditions on the moments of order $0$ and $1$) of the Laplacians acting on functions over general domains $\Omega$.
\end{rem}

\section{Well-posedness of the parabolic problem in the space of zero mean functions}\label{sec:parab}

In the previous section we have showed that the densely defined sesquilinear, symmetric, form $a$ is bounded and coercive both on $\tilde{V}$ and whenever restricted to $V$. Hence, from the general theory of forms we deduce that the associated operators $\tilde{A}$ and $A$ generate a cosine operator function on $H$.
Since we are especially interested in diffusion-type problems, we want to mention explicitly the following immediate consequence, where we are using the compact embedding of $H^1\left( (0,1)^N \right)$ in $L^2\left( (0,1)^N \right)$.

\begin{prop}\label{generationinh}
Both operators $-A$ and $-\tilde{A}$  generate  analytic, contractive, exponentially stable semigroups.
\end{prop}

Hence, we can say that the abstract Cauchy problems associated with $A$ and $\tilde{A}$ are well-posed. However, the challenge is now to understand just \emph{which} are these abstract Cauchy problem. The characterization of $D(A)$ and $D(\tilde{A})$ in Theorem~\ref{identk=0} and Theorem~\ref{identVk=0} is not very satisfying -- and correspondingly poor is the description of the differential equations effectively solved by the semigroups. Yet, we are able to deduce well-posedness of the classical heat equation in the following weak sense.

\begin{theo}
Let $u_0 \in D(A)$. 
Then the function
\[
t\mapsto u(t,\cdot):=e^{-tA}u_0(\cdot)
\]
solves 
\[
\frac{\partial u}{\partial t}(t,x)=\Delta u(t,x),\qquad u(0)=u_0,
\]
with conditions
\begin{equation}\label{condint}
\int_{(0,1)^N} u(t,y)\bar h(y)dy=0\quad \forall h\in {\rm Har}\left( (0,1)^N \right),\qquad  t> 0,
\end{equation}
in the sense of distributions, i.e., for all $t>0$~\eqref{condint} is satisfied and moreover
\[
\left\langle \frac{\partial u}{\partial t}(t,\cdot)-\Delta u(t,\cdot),g \right\rangle_{H-H'}=0 \qquad \hbox{for all }g \in H^1_0\left((0,1)^N\right).
\]
\end{theo}

\begin{proof}
It suffices to take into account Theorem~\ref{identVk=0} and recall that the analytic semigroup $(e^{-tA})_{t\ge 0}$ maps $u_0$ into the domain of $D(A)$ for all $t>0$, and moreover 
$u$ solves 
\[
	{\left\langle  \frac{\partial u}{\partial t}(t,\cdot)-{\rm Id}_m^{-1}\Delta u(t,\cdot), g \right\rangle_{H-H'} \qquad \hbox{for all }g \in H',}
\]
Now, ${\rm Har}(T^N)$ is a closed subspace of $H^1\left( T^N\right)$ and $H$ is its annihilator, we can deduce that there is an isomorphism
\[
\mathfrak J:H^1\left( T^N\right)/ {\rm Har}(T^N)\to H',
\]
given by
\[
\langle \varphi, \mathfrak J [h] \rangle_{H-H'}:=\langle \varphi, h\rangle_{H^{-1}(T^N)-H^1\left( T^N\right)}, \forall \varphi\in H.
\]
Then we have that
\[
\left\langle  \frac{\partial u}{\partial t}(t,\cdot)-{\rm Id}_m^{-1}\Delta u(t,\cdot),g \right\rangle_{H-H'}=0 \qquad \hbox{for all }g \in H^1\left( T^N\right),
\]
and in particular
\[
\left\langle  \frac{\partial u}{\partial t}(t,\cdot)-{\rm Id}_m^{-1}\Delta u(t,\cdot),g \right\rangle_{H-H'}=0 \qquad \hbox{for all }g \in {\mathcal D}\left((0,1)^N\right).
\]
Now, the assertion follows from Remark~\ref{idmdl}. 
\end{proof}

Similarly, exploiting instead Theorem~\ref{identk=0}, we can obtain the following.

\begin{theo}
Let $u_0 \in D(\tilde{A})$. Then the function
\[
t\mapsto u(t,\cdot):=e^{-t\tilde{A}}u_0(\cdot)
\]
solves 
\[
\frac{\partial u}{\partial t}(t,x)=\Delta u(t,x),\qquad t>0, \qquad u(0)=u_0,
\]
in the sense of distributions, it is periodic in the sense of Definition~\ref{def:weaklyper} and satisfies conditions
\begin{equation}\label{condint2}
\int_{(0,1)^N} u(t,y)\bar h(y)dy=0\quad \forall h\in {\rm Har}(T^N),\qquad  t> 0.
\end{equation}
\end{theo}

Observe that in particular the constant functions belong to ${\rm Har}(T^N)$, and therefore $H$ contains all mean zero functions -- i.e., all functions with vanishing moment of order 0. Thus, if $f\in \tilde{V}$ (and in particular if
$f\in D(\tilde{A})$), then 
\[
\int_{(0,1)^N} 	f(x) dx=0.
\]
Furthermore, the functions $g_i:(0,1)^N\ni x=(x_1,\ldots,x_n)\mapsto x_i\in (0,1)$ belong to ${\rm Har}\left( (0,1)^N \right)$ for all $i=1,\ldots,N$. Hence: If additionally $f\in V$ (and in particular if
$f\in D(A)$), then 
\[
\int_{(0,1)^N} 	f(x) dx=\int_{(0,1)^N}x_i 	f(x) dx=0,\qquad i=1,\ldots,N.
\]
In particular,  for all $t>0$ both $(e^{-tA})_{t\ge 0}$ and  $(e^{-t\tilde{A}})_{t\ge 0}$ map any $u_0\in H$ into a function that has mean zero. Furthermore $(e^{-tA})_{t\ge 0}$ maps any $u_0\in H$ into a function that has vanishing \emph{first linear moments along each axis}, as such functions are referred to in~\cite[\S~9.6.5]{Vaz06}.

\section{$L^2$-results and interplays with extension theory}\label{sec:krein}

It is well-known that if an operator that comes from a sesquilinear form generates an analytic semigroup, then so does its part in the form domain. In this conclusive section we will devote our attention in particular to the part in $V$ of the operator $A$ described in Theorem~\ref{identVk=0}.
To begin with, we recall a celebrated result due to M.\ Krein: Each closed, densely defined, symmetric, positive definite operator has \emph{one} smallest self-adjoint positive definite extension\footnote{ In the sense of forms: Let $a_1,a_2$ be two densely defined symmetric bounded elliptic forms with associated operators $A_1,A_2$. If their form domains $D(a_1),D(a_2)$ satisfy $D(a_2)\subset D(a_1)$ and $a_1(x,x)\le a_2(x,x)$ for all $x\in D(a_2)$, then $A_1$ is said to be \emph{smaller than} $A_2$.}. This extension is nowadays commonly called the \emph{Krein--von Neumann extension} -- cf.~\cite{AshGesMit10} for a brief introduction to this subject including the connection with the so-called ``buckling problem'' of mathematical physics, or~\cite{AshGesMit13} and~\cite[Chapter 13]{Sch12} for very comprehensive overviews.

{
It turns out that such an extension can be also characterized as follows -- this is but a special instance of~\cite[Thm.~2.7]{AshGesMit13}.
\begin{lemma}
Let $S$ be a closed, densely defined, symmetric operator on a Hilbert space $\mathcal H=L^2(\Omega)$ such that $S-\epsilon\Id$ is positive definite for some $\epsilon>0$. Among all its self-adjoint extensions there exists exactly one whose domain contains the null space of $S^*$: This is precisely the Krein--von Neumann extension.
\end{lemma}Now, let $S$ be the Laplacian on $L^2(T^N)$ defined on test functions only,
\[
Sf:=-\Delta f,\qquad f\in {\mathcal D}((0,1)^N)
\]
Then, $S$ is  clearly symmetric (but of course not self-adjoint) and negative definite, and in fact it follows from the Poincar\'e inequality that $S-\epsilon\Id$ is positive definite for some $\epsilon>0$ (for example, if one takes $\epsilon$ to be the lowest eigenvalue of the Laplacian with Dirichlet boundary conditions). It follows from our construction that the domain of the part  $A_V$ of $A$ in the form domain $V$, i.e.,
\[
D(A_V)=\{v\in V:Av\in V\}
\]
is orthogonal to the null space of $S^*$ (that is, to the space of (weakly) harmonic functions.}

Let $A_K$ be the extension of $A_V$ defined by
$$
D(A_K)=D(A_V)\oplus V_1,
$$
and
$$
A_K v= A_V v_0, \quad\forall v=v_0\oplus v_1\in D(A_V)\oplus V_1.
$$
Hence $A_K$ is a selfadjoint operator in $L^2((0,1)^N)$, and therefore $-A_K$ is exactly the Krein--von Neumann extension of $S$. 

We thus obtain the following alternative characterization of the Laplacian under conditions on the moments of order 0 and 1.
\begin{theo}
The $L^2$-realization $-A_V$ of the Laplacian agrees with the \emph{reduced Krein--von Neumann} Laplacian as introduced in~\cite[\S~2]{AshGesMit13}. In particular, 
\[
A_V f=\Delta f,\qquad \forall f\in D(A_V),
\]
(and not only $A_V f=\Id_m^{-1}\Delta f$!).
\end{theo}

Furthermore, because $\Omega=T^N$ is a bounded Lipschitz domain (hence a \emph{quasi-convex domain} in the sense of~\cite[\S~5.1]{AshGesMit13}) one can apply the general theory surveyed in~\cite{AshGesMit13} and in particular by~\cite[Thm~6.5]{AshGesMit13} $-A_K$ is precisely the Laplacian 
 with boundary conditions given by
\[
\frac{\partial u}{\partial \nu}={\rm D\!\! N} u_{|\partial \Omega}\qquad \hbox{on }\partial \Omega,
\]
where ${\rm D\!\! N}$ is the Dirichlet-to-Neumann operator associated with $\Delta$ on $\partial \Omega$,  defined in a suitably weak sense, see~\cite[\S~5--6]{AshGesMit13}.

\begin{rem}
In case of $N=1$ we recover in particular the boundary conditions 
\[
u'(1)=u'(0)=u(1)-u(0):
\] 
This in accordance with~\cite[Cor.~3.10]{MugNic13} and also with the setting of~\cite{BobMug13}. The same boundary conditions are also referred to as \emph{transparent conditions} in numerical analysis and \emph{$\delta'$-interaction} in mathematical physics, cf.\ respectively~\cite{HalRau95,Exn96} for an interpretation of these conditions and a list of related references. While the Krein--von Neumann extension is already known to have interesting connections with elasticity theory, cf.~\cite{AshGesMit13}, the observation that the associated heat equation enjoys conservation of moments of order 0 and 1 seems to be new in the literature.
\end{rem}

\bibliographystyle{abbrv}
\bibliography{../../referenzen/literatur}

\end{document}